\documentclass[12pt]{amsart}
\usepackage{cite}
\usepackage{amsfonts,enumerate,amsmath,amsthm,amssymb,cite,mathrsfs}
\usepackage{graphicx}
\usepackage{caption}
\usepackage{subcaption}
\usepackage{fancyhdr}
\theoremstyle{definition}
\newtheorem{definition}{Definition}[section]
\newtheorem{remark}[definition]{Remark}
\newtheorem{example}[definition]{Example}
\theoremstyle{plain}
\newtheorem{proposition}[definition]{Proposition}
\newtheorem{theorem}[definition]{Theorem}

\newtheorem{lemma}[definition]{Lemma}

\title{Ergodic Decomposition}

\author{Sakshi Jain\textsuperscript{1}}

\address{1. Department of Mathematics, University of Delhi, New Delhi-110007}
\email{sjsakshi29jain@gmail.com}
\author{Shah Faisal\textsuperscript{2}}
\address{2.Berlin Mathematical School (BMS), Berlin, Germany and \linebreak International Center for Theoretical Physics (ICTP), Trieste, Italy}
\email{faisal@math.hu-berlin.de}

\thanks{\textsuperscript{1} We thank Stefano Luzzatto for suggesting this topic and for guiding us throughout this project and we thank Vilton Pinheiro for pointing out the possibility of the result that we have proved and for suggesting a strategy for it. Also we thank Oliver Butterley for his helpful suggestions and careful reading of the script.  
\\ \textsuperscript{2} Sakshi Jain was supported by CSIR- Junior Research Fellowship (File No.- 09/045(1522)/2017-EMR-I) of Goernment of India. Shah Faisal was supported by ICTP postgraduate Diploma Scholarship.}

\begin{document}

\begin{abstract}
Ergodic systems, being indecomposable are important part of the study of dynamical systems but if a system is not ergodic, it is natural to ask the following question:
\begin{quote}
	Is it possible to split it into ergodic systems in such a way that the study of the former reduces to the study of latter ones?
\end{quote}
Also, it will be interesting to see if the latter ones inherit some properties of the former one. This document answers this question for measurable maps defined on complete separable metric spaces with Borel probability measure, using the Rokhlin Disintegration Theorem.
\end{abstract}

\maketitle

\noindent
\section{Introduction and statements of results}

Let \(  (X, \mathcal A, \mu)  \) be a probability space where \(\mathcal A\) has a countable generator and   $\mathbb{P}$  a partition  of $X$ into measurable subsets. The basic question we address is the following:

\begin{quote}
	Is it possible to ``disintegrate''  \(  \mu  \) into ``conditional'' measures on the elements of the partition \(  \mathbb P  \)?
\end{quote}
Under certain conditions on the system under consideration, the answer to this question is affirmative. For instance, if \(f:X\to X\) is continuous and \(X\) is compact then every \(f-\)invariant measure is a convex combination of ergodic invariant measure, by Krylov-Bogolyubov Theorem \cite{W} together with Choquet's Theorem. However, if the map \(f:X\to X\) is not continuous and$/$or the space \(X\) is not compact then the set \(\mathcal M\) of all probability measures on the system under consideration may not be compact and in that case we are unable to apply the Choquet's Theorem.
In this document, we present a similar result for a relatively larger class of maps defined on a reasonably good space.
Moreover, we try to answer in here that, if \(  \mu  \) has some additional properties, such as being non-singular or invariant with respect to some measurable transformation \(  f: X \to X  \) can we construct such a disintegration so that the conditional measures inherit these properties?

The question of existence of disintegration of a measure on a Borel probability space has been addressed originally by Rokhlin in \cite{RO}. It is a key tool in proving the existence of ergodic decomposition in this document. Ergodic decomposition is quite fundamental and often quoted result but there does not exist a detailed presentation of it. Here, we have made an effort to address the topic in completely detailed manner. It should be mentioned that the ergodic decomposition for invariant measure has been proved in \cite{W}. Here we prove the existence of ergodic decomposition with reduced hypothesis and include various examples illustrating the main definitions and results. The proof is not particularly complicated but neither is it trivial, and this generalization may be useful in the study of ergodic properties of non-invariant measures. We believe that this document can be a reference for this topic.

To formulate the question of ergodic decomposition precisely,  we first define the following notions.
We can define a canonical projection
\[
\widehat{\tau}:X\rightarrow \mathbb{P} \quad \text{ by } \quad \widehat{\tau}(x)=P(x)
\]
where $P(x)\in \mathbb{P}$ is such that $x\in P(x)$. Then  we can define a sigma-algebra $\widehat{\mathcal{A}}$ of measurable sets on \(  \mathbb P  \) by
defining a subset $A\subset \mathbb{P}$ to be measurable if and only if $\widehat{\tau}^{-1}A$ is a measurable subset of $X$.  Notice that each element \(  P\in \mathbb P  \) belongs to $\widehat{\mathcal{A}}$.
We can then define the \emph{quotient measure} $\widehat{\mu} $ on $\mathbb{P}$ by letting
\begin{equation*}
\widehat{\mu}(A)=\mu(\widehat{\tau}^{-1}(A))
\end{equation*}
for all \( A\in \widehat{\mathcal{A}}\).

\subsection{Disintegration of measures}
We can now formally define the notion of disintegration of \(  \mu  \).

\begin{definition}[Disintegration of a measure]\label{dis}
	Given a partition \(  \mathbb P  \) of \(( X, \widehat{\mathcal A}, \mu )\) into measurable subsets, a  family of probability measures $\{\mu_P: P\in \mathbb{P}\}$ on $X$ is said to disintegrate $\mu$ with respect to $\mathbb{P}$ if the following hold:
	\begin{enumerate}
		\item $\mu_P(P)=1$ for $\widehat{\mu}$-almost every $P\in \mathbb{P}$.
		\item For every measurable subset $E$ of $X$, the map $P\rightarrow\mu_P(E)$ is measurable and
		\[
		\mu(E)=\int_{\mathbb{P}}\mu_P(E)d\widehat{\mu}(P).
		\]
	\end{enumerate}
	We call the measures \(  \mu_{P}  \) \emph{conditional measures} of \(  \mu  \) with respect to \(  \mathbb P  \).
\end{definition}
The first condition ensures that the conditional measures $\mu_{P}$ are mutually singular, and thus independent, in some sense. The measurability and boundedness of the map $P\to \mu_{P}(E)$ ensures that the integral in the definition of  $\mu$ exists.

\begin{remark} Notice that if $\mathbb{P}$ is finite or countable, or more generally if there exists a finite or countable set $ \{P_i\} $ of elements of \( \mathbb P  \) such that  \(  \mu (\cup P_{i})=1 \) (throwing away the measure zero elements of \(\mathbb P\)), the integral reduces to a  convex combination of $\mu_{P}$:
	\begin{equation*}
	\mu(E)=\int \mu_{P}(E)\ d \widehat{\mu}(P)
	=\sum\mu_{P_i}(E)\widehat{\mu}(P_i),=\sum \mu_{P_i}(E)\mu(P_i).
	\end{equation*}
	Notice that  $\Sigma \mu(P_i)=1$. In this case, we have in fact a very explicit form for the conditional measures which can be defined as
	\[
	\mu_{P_i}(E)=\frac{\mu (E\cap P_i)}{\mu(P_i)}
	\]
	for all measurable sets E.
\end{remark}

More interesting and non-trivial situations are when none of the partition elements of \(  \mathbb P  \) has positive measure.

\begin{example}
	Let \(\mathbb P\) be the partition of \([0,1]\) into singleton sets that is, \(\mathbb P =\{\{x\}: x\in [0,1]\}\) and \(m\) the Lebesgue measure on \( [0,1]\). Clearly, \(m(P)=0 \) for all \(P\in \mathbb P\).   For each \(  P_{x}\in\mathbb P  \) with \(  P_{x}=\{x\}  \), let \(  \mu_{P_{x}}:=\delta_{P_{x}}  \) where \(  \delta_{P_{x}}  \) is the Dirac-delta probability measure on the point \(  x  \).  We claim that the family \(  \{\mu_{P}\}  \) is a disintegration of \(  m  \) with respect to \(  \mathbb P  \).
	Indeed, for each
	\(P\in \mathbb P\) we clearly have \(\mu_{p}(P)=\delta_P(P)=1\). Also,   for any measurable set \(E\subset X\), the map \(P \to \mu_P(E)\) takes values in \(  \{0,1\}  \) and the pre-image of \(  1  \) is the collection of all partition elements \(  P_{x}\in \mathbb P  \) with \(  x\in E  \) which is therefore measurable by the measurability of \(  E  \) and the definition of the sigma-algebra on \(  \mathbb P  \) defined by the quotient map \(  \widehat \tau  \).  Therefore for any measurable set \(  E\subseteq X  \), we have
	\begin{equation*}
	\int_{\mathbb P} \mu_P(E)d\widehat{\mu}(P) = \int_{\mathbb P} \delta_P(E)d\widehat\mu(P) = \widehat\mu (\widehat\tau (E)) = m(E).
	\end{equation*}
	As a similar but slightly more general example,   let \(  X=  [-1,1]\)  with normalised Lebesgue measure \(m\),  the partition  \(\mathbb P=\{\{-x,x\},x\in(0,1]\} \cup \{0\}\). We claim that the family of probability measures \(\{\mu_P=\delta_x/2+\delta_{-x}/2: P=\{-x,x\}\}\cup \{\delta_0\}\) is the disintegration of \(m\) with respect to \(\mathbb P\). Observe that   \(\mu_P(P) = 1\) for all \(P\in \mathbb P\), by similar reasoning as above,  for measurable set \(E\subset X\) the map \(P\to \mu_P(E)\)  is measurable, and
	\begin{equation*}
	\int_{\mathbb P} \mu_P(E)d\widehat{\mu}(P) =   \frac{1}{2} \int_{\mathbb P} (\delta_x(E)+ \delta_{-x}(E))d\widehat{\mu}(P) = m(E).
	\end{equation*}
\end{example}

\subsection{Uniqueness of disintegration}\label{sec:unique}
With the formal definition of disintegration of a measure we can formalize the question of the existence and uniqueness of such a disintegration.

\begin{definition}[Uniqueness of Disintegration]
	A probability measure \(  \mu  \) is said to have a \emph{unique}
	disintegration with respect to the partition \(  \mathbb P  \) if given any two disintegrations $\{\mu_P:P\in \mathbb{P}\}$ and $\{{\mu}^{\prime}_P:P\in \mathbb{P}\}$, we have  $\mu_P={\mu}^{\prime}_P$ for $\widehat{\mu}$-almost every $P\in\mathbb{P}$.
\end{definition}

Both the existence and uniqueness of a disintegration of a measure \(  \mu  \) with respect to a partition \(  \mathbb P  \) are in general non-trivial.  The uniqueness however follows immediately (using the fact that our probability space has a countable generator) without any additional assumptions on the partition.

\begin{proposition}[Uniqueness of Disintegration\cite{W}]\label{thm:unique}
	Let \((X,\widehat{\mathcal A}, \mu)\) be a probability space such that the sigma-algebra \(\widehat{\mathcal A}\) has a countable generator and \(\mathbb P\) be a partition of \(X\), then the disintegration of \(\mu\) into conditional measures, if it exists, is unique.
\end{proposition}

Proposition \ref{thm:unique}   allows us to give an example of a measure which does not admit any disintegration with respect to a given partition.

\begin{example}\label{examplekhan}
	
	Let $X=S^1$ with Lebesgue measure $m$ and  $f:S^1\to S^1$ be an irrational rotation \(  f(x)= x+\alpha   \) mod 1 for some irrational number \(  \alpha  \). Let \(  O(x) := \{f^{n}(x): {n\in \mathbb Z} \}  \) be the full orbit of the point \(  x  \)
	and let  $\mathbb{P}:=\{P_x=O(x): x\in S^1\}$.
	We claim that there does not exist any disintegration of $m$ with respect to the partition \(\mathbb P\). Supposing the contrary, let $\big\{\mu_{P_x}: P_x\in \mathbb{P}\big\}$ be a disintegration of $m$ with respect to $\mathbb{P}$. We prove that the family of the pull-backs (defined as \(f_\star\mu(A)=\mu(f^{-1}(A))\), for all \(A \in X\))  of $\mu_{P_x}$ by $f$, $\big\{f_\star\mu_{P_x}: P_x\in \mathbb{P}\big\}$ is a also a disintegration of $m$ with respect to $\mathbb{P}$.
	\begin{enumerate}
		\item $\mu_{P_x}(P_x)=1$ for $\widehat{\mu}$--almost every $P_x\in \mathbb{P}$ implies $f_\star\mu_{P_x}(P_x)=\mu_{P_x}(f^{-1}(P_x))=\mu_{P_x}(P_x)=1$ for $\widehat{\mu}$--almost every $P_x\in \mathbb{P}$.
		\item For every measurable subset $E$ of $Y$, the map $P_x\to \mu_{P_x}(E)$ is measurable implies $P_x\to f_\star\mu_{P_x}(E)$ is also measurable.
		\item For every measurable subset $E$ of $Y$, by the invariance of $m$, we have
		\[m(E)=m(f^{-1}(E))={\displaystyle \int_{\mathbb{P}}} \mu_{P_x}(f^{-1}(E))\ d \widehat{\mu}(P)={\displaystyle \int_{\mathbb{P}}} f_\star\mu_{P_x}(E)\ d \widehat{\mu}(P).\]
	\end{enumerate}
	This proves that $\big\{f_\star\mu_{P_x}: P_x\in \mathbb{P}\big\}$ is a disintegration of $m$ with respect to $\mathbb{P}$. By Proposition \ref{thm:unique}, we have $f_\star\mu_{P_x}=\mu_{P_x}$  for $\widehat{\mu}$--almost every $P_x\in \mathbb{P}$. Thus $\mu_{P_x}$ is $f$-invariant for $\widehat{\mu}$--almost every $P_x\in \mathbb{P}$. Since
	Lebesgue measure $m$ is the only invariant measure (because it is an irrational rotation), so $\mu_{P_x}=m$  for $\widehat{\mu}$--almost every $P_x\in \mathbb{P}$. This is a contradiction because $m(P_x)=0$ and $\mu_{P_x}(P_x)=1$ for $\widehat{\mu}$--almost every $P_x\in \mathbb{P}$. Thus there does not exist any disintegration of $m$ with with respect to the partition into orbits, $\mathbb{P}$.
\end{example}

\begin{proof}[Proof of Proposition \ref{thm:unique}]  We know that $\widehat{\mathcal{A}}$ has a countable generator, say $U$. Let $A_U$ denotes the algebra generated by $U$. Let if possible \(\mu\) has two disintegrations with respect to \(\mathbb P\) namely, \(\{\mu_P: P \in \mathbb P\}\) and \(\{\mu'_P:P \in \mathbb P\}\).
	
	It is enough to prove that $\widehat{\mu}(A_{E})=\widehat{\mu}(B_{E})=0$ for every $E\in A_{U}$, where
	\[A_{E}=\Big\{P\in \mathbb{P}: \mu_{P}(E)< \mu'_{P}(E)\Big\} \text{ and } B_{E}=\Big\{P\in \mathbb{P}: \mu_{P}(E)>\mu'_{P}(E)\Big\}.\]
	Note that if $P\in A_{E}$ then $P\subseteq \widehat{\tau}^{-1}(A_{E})$ and therefore $\mu_{P}(E\cap \widehat{\tau}^{-1}(A_{E}))=\mu_{P}(E)$ for $\widehat{\mu}$-almost every $P\in \mathbb{P}$, otherwise $\mu_{P}(E\cap \widehat{\tau}^{-1}(A_{E}))=0$. The same is true for $\mu'_P$. Moreover,
	\[
	\mu(E\cap \widehat{\tau}^{-1}(A_{E}))=
	\begin{cases}
	{\displaystyle\int_{\mathbb{P}}}\mu_{P}(E\cap \widehat{\tau}^{-1}(A_{E}))d\widehat{\mu}(P)={\displaystyle\int_{A_E}}\mu_{P}(E)d\widehat{\mu}(P)&\\
	\vspace{.2cm}\\
	{\displaystyle\int_{\mathbb{P}}}\mu'_{P}(E\cap \widehat{\tau}^{-1}(A_{E}))d\widehat{\mu}(P)={\displaystyle\int_{A_E}}\mu'_{P}(E)d\widehat{\mu}(P)&.
	\end{cases}\]
	The implies that
	\[{\displaystyle\int_{A_E}}(\mu_{P}(E)-\mu'_{P}(E))d\widehat{\mu}(P)=0\]
	which leads to $\widehat{\mu}(A_E)=0$ because $\mu_{P}(E)-\mu'_{P}(E)>0$. Similarly, we can prove $\widehat{\mu}(B_E)=0$. Since $A_{U}$ is countable, so
	\[\widehat{\mu}(\cup_{E\in A_{U}}(A_E \cup B_E))=0.\]
	This proves that for almost every $P$, $\mu_{P}$ and $\mu'_{P}$ agree on the algebra $A_U$, hence they agree on the sigma algebra, $\mathcal{A}$, generated by $A_U$.
\end{proof}

\subsection{Existence of disintegrations}
To guarantee the existence of a disintegration we need some further assumptions both on the probability space and on the partition.
Before giving the next definition we recall   that  a partition \(\mathbb P_2\) of \(X\) is a refinement of a partition \(\mathbb P_1\) of \(X\), denoted as \(\mathbb P_1\preceq \mathbb P_2\), if for every \(A_{2} \in \mathbb P_2\), there exists \( A_{1}\in \mathbb P_1\) such that \(  A_{2}\subseteq A_{1}  \).

\begin{definition}[Measurable Partitions]\label{mp}
	A partition $\mathbb{P}$ into measurable subsets of \(  X  \) is a  \(  \mu  \)-\emph{measurable partition}  if there exists a subset $X_0\subset X$ of full measure and a sequence of countable partitions $\mathbb{P}_n$, each consisting of measurable sets, such that $\mathbb{P}_n\preceq \mathbb{P}_{n+1}$ for all $n\in \mathbb{N}$ and
	every point in \(  X_{0}  \) can be written as a countable intersection of partition elements \(  P_{n}\in \mathbb P_{n}  \).
	We   refer to \(  \mathbb P  \) as  a \emph{measurable partition} for  \(  \mu  \).
\end{definition}

\begin{example}[Countable partitions are measurable] \label{sens}
	Notice that every countable partition is measurable with respect to any probability measure, take \(\mathbb P_n=\mathbb P\) for all \(n \in \mathbb N\). On the other hand,  in general, the measurability of a partition  depends upon the measure under consideration.
	For example let $(X,\mathcal{A}, \mu)$ be a probability space and $f:X\to X$ be a measurable transformation such that $\mu$ is ergodic (see Definition \ref{def:ergodic} below). We claim that the partition $\mathbb{P}:=\{P_x=O(x): x\in S^1\}$ into orbits is measurable with respect to $\mu$ if and only if there exists an orbit of full measure, which implies that the measurability of the partition depends on the measure.  Assume that $\mathbb{P}$ is measurable, then by definition there exists a sequence $\mathbb{P}_n$ of countable partitions such that $\mathbb{P}_n\preceq \mathbb{P}_{n+1}$ for all $n\in \mathbb N$ and for every $P_x\in \mathbb{P}$ there exists a sequence $P_x^n\in \mathbb{P}_n$ such that $P_x=\cap_{n\in \mathbb N}P_x^n$. This means that, for every $n\in \mathbb N$, each $P_x^n \in \mathbb{P}_n$ is a union of orbits and hence invariant under $f$, that is, $f^{-1}(P_x^n)=P_x^n$. By ergodicity either $\mu(P_x^n)=1 \text{ or } 0$. So for each $n\in \mathbb N$, there exists $P_x^n \in \mathbb{P}_n$ such that $\mu(P_x^n)=1$. The orbit corresponding to $\cap_{n\in \mathbb N}P_x^n$ has full measure. The other way around is trivial.
\end{example}

\begin{example}[An uncountable measurable partition]\label{ex:measpart}
	Consider the torus $\mathbb T^2 = S^1\times S^1$, endowed with the Lebesgue measure $m$ and the partition $\mathbb{P}=\{x\times S^1: x\in S^1\}$.	 For each $n\in \mathbb N$ , define $\mathbb{P}_n$ by
	\[\mathbb{P}_n=\{J(i,n)\times S^1: i\in \{1,2,3,\dots, 2^n\}\}, \text{ where } J(i,n)=\left[\frac{i-1}{2^n},\frac{i}{2^n}\right).\]
	Clearly, each $\mathbb{P}_n$ is finite and $\mathbb{P}_n\preceq \mathbb{P}_{n+1}$ for all $n\in \mathbb N$. Since $J(i,n)$ is a partition of $[0,1)$, for every $x\in [0,1)$ and $n\in \mathbb N$ there exists $i\in \{1,2,3,\dots, 2^n\}$ such that $x\in J(i,n)$. Clearly, $\overline{J(i,n)}$ defines a sequence of closed intervals whose diameter goes to $0$ as $n\to \infty$. So by Cantor's Intersection Theorem, the intersection of all $\overline{J}(i,n)$  contains just $x$. Therefore,
	\[x\times S^1=\bigcap_{n\in \mathbb N} J(i,n)\times S^1
	\]
	and so   $\mathbb{P}$ is a measurable partition.
\end{example}	

\begin{example}[A Non-measurable Partition]
	Consider the two torus $\mathbb T^2 = S^1\times S^1$, endowed with the Lebesgue measure $m$.
	Define an Anosov diffeomorphsim $f: \mathbb T^2 \to \mathbb T^2$ by the integer matrix
	\[\left(
	\begin{array}{cc}
	2 & 1 \\
	1 & 1 \\
	\end{array}
	\right)\text{mod}\ 1.\]
	Notice that \((0,0)\) is a fixed point and
	the eigen values of this matrix are \(\lambda=\frac{(3+\sqrt{5})}{2} > 1\) and \(1/\lambda\). The eigen vector corresponding to the eigenvalue \(\lambda\) is \(((1+\sqrt5)/2, 1)
	\). The eigenspace is a line which is the unstable manifold which wraps around the torus.
	
	Let $\mathbb{P}=\{\mathcal{W}^u(x): x\in \mathbb T^2\}$ be the partition of $\mathbb T^2$ into the unstable manifolds.
	We prove that $\mathbb{P}$ is not measurable. Assume that $\mathbb{P}$ is measurable, then by definition there exists a sequence $\mathbb{P}_n$ of countable partitions such that $\mathbb{P}_n\preceq \mathbb{P}_{n+1}$ for all $n\in \mathbb N$ and for every $ \mathcal{W}^u(x) \in \mathbb{P}$ there exists a sequence $P_n\in \mathbb{P}_n$ such that $\mathcal{W}^u(x)=\cap_{n\in \mathbb N}P_n$. This means, for every $n\in \mathbb N$, each $P_n \in \mathbb{P}_n$ is a union of unstable manifolds and hence invariant under $f$, that is, $f^{-1}(P_n)=P_n$. Observe that assuming \(c_1=1, c_2=(1+\sqrt5)/2\), then \(c_2/c_1\) being irrational gives an irrational flow on \(\mathbb{T}^2\) of the form \(\phi^t(x_1,x_2)=(x_1+c_1t,x_2+c_2t)\) mod \(\mathbb Z^2\) and the partition $\mathbb{P}$
	corresponds to the partition into orbits by this irrational flow, so $m$ is ergodic. By ergodicity either $m(P_n)=1 \text{ or } 0$. So for each $n\in \mathbb N$, there exists $P_n \in \mathbb{P}_n$ such that $\mu(P_n)=1$. The $\mathcal{W}^u(x)=\cap_{n\in \mathbb N}P_n$ has full measure which is absurd.
\end{example}

The following classical result, which we will prove in Section \ref{sec:rokhlin},  gives conditions for the existence of a disintegration of a probability measure with respect to a partition.

\begin{theorem}[Rokhlin Disintegration, \cite{RO}]\label{Rokhlin}
	Let  \(  X  \) be a separable metric space, \(  \mu  \) be a Borel probability measure on \(  X  \), and \(\mathbb{P}\) a measurable partition. Then \(  \mu  \) admits a (unique) disintegration with respect to $\mathbb{P}$.
\end{theorem}

\begin{example}[Disintegration with respect to a measurable partition]\label{ex:disint}
	Recall the   partition $\mathbb{P}=\{x\times S^1: x\in S^1\}$   of $\mathbb T^2 = S^1\times S^1$  of Example \ref{ex:measpart} which we proved was a measurable partition.
	Let $m_x$ be the Lebesgue measure on the fiber $x\times S^1$ measuring arc length. By the Fubini's Theorem, 	for every measurable set $E$, we have
	\begin{equation*}
	m(E)=\int_{S^1\times S^1}\chi_E \ dm =\int_{S^1}\left(\int_{S^1}\chi_E \ dm_{x}\right) \ dm_{y} =\int_{S^1} m_x(E) \ dm_{y}.
	\end{equation*}
	This proves that $\{m_x: x\in S^1\}$ disintegrates $m$ with respect to $\mathbb{P}$.
\end{example}

\subsection{Ergodic disintegration}

Let \(  (X, \widehat A, \mu)  \) be a probability measure space and
\(
f: X \to X.
\)
a measurable transformation.  Then there may be some relationship between the measure \(  \mu  \) and the map \(  f  \) and the main focus of this note is to study how these relationships may or may not be inherited  by the conditional measures for certain partitions. Recall that for a probability measure \(  \mu  \) we  define \(  f_{*}\mu (A):= \mu(f^{-1}(A)  \).

\begin{definition} [Ergodic, invariant and non-singular measure]\label{def:ergodic}
	Let \(  X  \) be a measure space and \(  \mu  \) a probability measure on \(  X  \).
	\begin{enumerate}
		\item
		\(\mu\) is  ergodic if \(E\in \widehat{\mathcal A}\) and \(f^{-1}(E)=E\) implies  \(\mu(E)=0\) or \(\mu(E)=1\).
		\item
		\(\mu\) is non-singular  if   \(  f_{*}\mu \ll \mu  \);
		\item
		\(\mu\) is invariant if   \(  f_{*}\mu=\mu  \).
	\end{enumerate}
\end{definition}

If \(  \mu  \) is \emph{not} ergodic with respect to \(  f  \) then we can decompose \(  X  \) into the union of two completely invariant non-trivial sets \(  X=A\cup A^{c}  \) with \(  f^{-1}(A)=A\) and \( f^{-1}(A^{c})=A^{c}  \) which means that there are essentially two distinct dynamical systems given by \(  f  \). In principle there is no reason why the measure \(  \mu  \) restricted to either \(  A  \) or \(  A^{c}  \) should be ergodic, and if it is not then we can repeat the argument to further decompose the space \(  X  \) into non-trivial fully invariant sets. A natural question is whether we can  write \(  \mu  \) as a combination of ergodic measures, or, more precisely, using the language above,
if there exists a partition \(  \mathbb P  \) of \(  X  \) into fully invariant measurable sets such that the conditional measures \(  \mu_{P}  \) given by the Rokhlin disintegration are all ergodic.

Our main result is the existence of an essentially canonical partition for which the conditional measures of the Rokhlin disintegration are always ergodic which has been proved. We show moreover that for this partition the non-singularity or invariance of the original measure \(  \mu  \) is always preserved in the disintegration.

\begin{definition}[Dynamical Partition]\label{dp}
	Let \(  (X,\widehat{\mathcal A}, \mu)  \) be a Borel probability space, let $\mathcal{A}$ be the algebra generated by the countable generator of $\widehat{\mathcal A}$, and let  $f:X\rightarrow X$ be  a measurable transformation.
	For $A \in \mathcal{A}$ and  $x \in X $,  let
	\begin{center}
		$\tau(x, A)= \liminf\limits_{n\rightarrow \infty}\frac{1}{n} \sharp \{0\leq i \leq{n-1} : f^i(x)\in A \}$
	\end{center}
	be the (liminf of the) asymptotic frequency of visits of the orbit of \(  x  \) to the set $A$.
	We define  the \emph{dynamical partition}  $\mathbb{P}_f$  of $X$ with respect to \(f\) as the partition into equivalence classes defined by the equivalence relation $x\backsim y$ if and only if $\tau(x, A)= \tau(y, A)$ for every $A \in {\mathcal{A}}$.
\end{definition}

\begin{example}	\label{ex:dynpar}
	Recall the   partition $\mathbb{P}=\{x\times S^1: x\in S^1\}$   of $\mathbb T^2 = S^1\times S^1$  of Examples \ref{ex:measpart} and \ref{ex:disint} which we proved is a measurable partition and for which we constructed the disintegration into conditional measures.
	Now, fix some \(  \alpha \in \mathbb Z  \) and define $T: \mathbb{T}^2 \to \mathbb{T}^2$ by
	\[T(x,y)=(x,y+\alpha x).  \]
	We claim that \(\mathbb P=\mathbb P_{T}\) is the dynamical partition with respect to \(T\).
	Indeed, note that any two $(x_1,y_1), (x_2,y_2)$ with $x_1\neq x_2$ do not belong to the same element of $\mathbb{P}$, because for $A,A'\in \mathcal{A}$ given by
	\[A=(0,q)\times S^1, \ A'=(q,1) \times S^1,\]
	where $q\in (x_1,x_2)$ is a rational, we have
	\[\tau((x_1,y_1),A')=\lim_{n\to \infty}\frac{1}{n}\sharp\{0\leq i \leq n-1: f^{i}(x)\in A'\}=0,\]
	and
	\[\tau((x_2,y_2),A')=\lim_{n\to \infty}\frac{1}{n}\sharp\{0\leq i \leq n-1: f^{i}(x)\in A'\}=1.\]
\end{example}

Now, having defined and stated all what is required, we are presenting the main result of this document.
\begin{theorem}[Ergodic Decomposition Theorem]\label{ergodic}
	Let \(  (X,\widehat{\mathcal A}, \mu)  \) be a Borel probability space where \(X\) is a separable metric space, $f:X\rightarrow X$  a measurable transformation and  $\mathbb{P}$ be the dynamical partition of $X$ with respect to \(f\).  Then \(  \mathbb P  \) is a measurable partition and for the Rokhlin disintegration $\{\mu_P\}$  of \(  \mu  \) with respect to \(  \mathbb P  \), we have, for $\widehat{\mu}$-almost every \(P\),  $\mu_P$ is ergodic. Moreover,if \(  \mu  \) is non-singular then for $\widehat{\mu}$-almost every \(P\),  $\mu_P$ is non-singular, and  if \(  \mu  \) is invariant, then for $\widehat{\mu}$-almost every \(P\), $\mu_P$ is invariant.
\end{theorem}
The existence of an ergodic disintegration for invariant measures can be found in \cite{W}.
\medskip

Before starting the proofs of our results we give a couple of examples of the construction of the dynamical partition and the conditional measures in some concrete cases.

\begin{example}
	
	Recall the   partition $\mathbb{P}=\{x\times S^1: x\in S^1\}$   of $\mathbb T^2 = S^1\times S^1$  of Examples \ref{ex:measpart} and \ref{ex:disint} which we proved was a measurable partition and is a dynamical partition with respect to \(T\) in Example \ref{ex:dynpar} and for which we constructed the disintegration into conditional measures which we now prove to be infact the ergodic decomposition of \(m\).
	
	\medskip 	
	
	Indeed,  for an irrational $x\in S^1$, the restriction of $T$ to the fiber $x\times S^1$ is an irrational rotation for which the Lebesgue measure  $m_x$ is well known to be ergodic. This proves that the family $\{m_x: x\in S^1\cap \mathbb {Q}'\}$ is an ergodic decomposition.

	\medskip
	Moreover, observe that the determinant of the Jacobian of $T$ is \(\alpha\) everywhere, that is, $|D(J(T))(x,y)|=\alpha$ for all $(x,y)\in \mathbb{T}^2$. Therefore for \(  \alpha\neq 1  \) Lebesgue measure is not invariant. It is however non-singular with respect to \(T\) and so are \(m_x\), for all \(x \in S^1\). In the case when \(\alpha =1\), \(m\) is invariant also and so are \(m_x\), for all \(x \in S^1\).

\end{example}
\begin{remark}
	Note that the disintegration with respect to the partition into horizontal fibers $\mathbb{P}=\{S^1\times y: y\in S^1\}$, is not an ergodic decomposition, note that $\mathbb{P}$ is not a dynamical partition.
\end{remark}
The examples below give the illustration of ergodic decomposition of an invariant measure into invariant ergodic measures.
\begin{example}
	Let $X=[0,1]$ and define $f:X\to X$ by $f(x)=x^2$. The partition $\mathbb{P}$ be defined as
	\[\mathbb{P}=\{P_1=\{0\}, P_2=\{1\}, P_3=(0,1)\}.\]
	Let $\mu$ be an invariant Borel probability measure. For $0<\epsilon <1$, by invariance, we have $\mu([0,\epsilon^n])=\mu([0,\epsilon])$ for all $n \in \mathbb N$. By the continuity of $\mu$, we have $\mu(\{0\})=\mu([0,\epsilon])$ which means $\mu((0,\epsilon])=0$. Take $\epsilon=1-1/2n$, we have $\mu((0,1-1/2n])=0$ for all $n\in \mathbb N$. By the continuity of $\mu$ we get $\mu((0,1))=0$, therefore $\mu(\{0\},\{1\})=1$. Note that $\widehat{\mu}(0,1)=0$, the family $\{\mu_{P_{1}}=\delta_0, \mu_{P_{2}}=\delta_{1}\}$ disintegrates every invariant Borel probability measure $\mu$, that is, $\mu=\mu\{0\}\delta_0+\mu\{1\}\delta_1$.
\end{example}
Now, having illustrated examples we move on to prove the results.

\section{The extension theorem} \label{prf ext & lem}
To prove Theorem \ref{Rokhlin}, we first need to prove the following very important theorem.
The following theorem ensures the extension of a finite additive function on the Borel algebra $\mathcal{A}$ to a countably additive function on the Borel sigma-algebra $\widehat{\mathcal{A}}$ where $X$ is a  completely separable topological metric space.

\begin{theorem}[Extension Theorem]\label{ext}
	Let \(  X  \) be a  separable metric space and let \(  \mathcal A\) and \(  \widehat{\mathcal A} \) be the Borel algebra and  the Borel sigma-algebra respectively of \(  X  \). Then every finite-additive function $\mu:\widehat{\mathcal{A}}\rightarrow [0,1]$ with $\mu(X)=1$ and $\mu({\emptyset})=0$, can be extended to a probability measure $\mu:\mathcal{A}\rightarrow [0,1].$
\end{theorem}
To prove this theorem, let us first define a few notations.
Let $U=\{S_k:k\in\mathbb{N}\}$ be the countable generator of the algebra $\mathcal{A}$ and the sigma-algebra $\widehat{\mathcal A}$ and let $\widehat{\mathcal{B}}$ and $\mathcal{B}$ be the Borel sigma-algebra and Borel algebra respectively generated by the cylindrical subsets of ${[0,1]}^{\mathbb{N}}$.
Consider the mapping $\varphi: X\to \{0,1\}^{\mathbb N}$ defined by
\[\varphi(x)=\{\mathcal{X}_{S_{k}}(x)\}_{k=1}^{\infty}, \ \text{for all} \ x \in X.\]
Notice that this map depends on $U$.  In the following text, for any set $A$, we give notation that $A^1$ and $ A^0$ means $A$ and $A^c$ respectively.
To prove Theorem \ref{ext} we first prove the following three lemmas.
\begin{lemma}\label{lemma1}
	The image $\varphi(X)$ is characterized by the following three properties, that is, $\{i_k\}_{k=1}^{\infty}\in \varphi(X)$ if and only if
	\begin{enumerate}
		\item For every $n \in \mathbb N$,
		\[\bigcap _{k=1}^{n}S_k^{i_k}\neq \emptyset .\]
		\item There exists $j\in \mathbb N$ such that $i_j=1$ and diameter$(S_j)\leq1$.
		\item For every $j\in \mathbb N$ such that $i_j=1$ there exists $l(j)\in \mathbb N$ with
		\[l(j)>j,\ \overline{S}_{l(j)}\subseteq S_j  \text{ and } \text{diameter}(S_{l(j)})\leq \text{diameter}(S_j)/2.\]
	\end{enumerate}
\end{lemma}
\begin{proof}``$\Rightarrow$'': If  $\{i_k\}_{k=1}^{\infty}\in \varphi(X)$ then for some $x\in X$, we have $\varphi(x)=\{i_k\}_{k=1}^{\infty}$ which implies that $x\in S_k^{i_{k}}$ for all $k \in \mathbb N$ and hence $\cap _{k=1}^{n}S_k^{i_k}$ contains $x$ for every $n \in \mathbb N$. This proves property (1).
	\newline
	Since $U$ is a basis, there exists some $S_l \in U$ with $x\in S_l$. So for $l$, we have $i_l=1$. Let $B(x,r)$ be a ball of radius  $r\leq1$ centered at $x$ such that  $B(x,r)\subseteq S_l$. Since $U$ is a basis, one can choose an $S_k\in U$ such that $x\in S_k\subseteq B(x,r)\subseteq S_l$. Clearly, we have $i_k=1$ and diameter($S_k$)$\leq1$. This proves (2).
	\newline
	Let $I$ be the set of all indices $l\leq k$ such that $x\in S_l$ and $\overline{S}_l\subseteq S_k$, where $k$ is an index satisfying (2). Let $r>0$ be such that $B(x,r)\subset S_l$ for all $i\in I$, such $r>0$ exists because $I$ is finite. Since $U$ is a basis, we can choose an $S_{l(k)}$ in $U$, for some $l(k)\in \mathbb N$, such that $\overline{S}_{l(k)}\subseteq B(x,r/2)\subseteq S_k$. Clearly $i_{l(k)}=1$, $l(k)>k$, $\overline{S}_{l(k)}\subseteq S_k$ and diameter($S_{l(k)}$)$\leq$diameter($S_k$)/2. This proves (3).
	\newline
	Conversely, let $\{i_k\}_{k\in \mathbb N}\in \{0,1\}^\mathbb N$ satisfies the above three properties. We show that there exists some $x\in X$ such that $\varphi(x)=\{i_k\}_{k\in \mathbb N}$. Define a sequence $\{A_k\}_{k\in \mathbb N}$ by
	\[A_k=
	\begin{cases}
	S_k^{i_k}& \text{if $i_k=0$}\\
	\overline{S}_{l(k)} & \text{if $i_k=1$}.
	\end{cases}\]
	Using this sequence we define another sequence $\{B_n\}_{n\in \mathbb N}$ as follows:
	\[B_n=\bigcap_{k=1}^{n}A_k.\]
	It is a decreasing sequence of non-empty closed sets that shrinks to a point. For each $n$, $B_n$ is non-empty because it contains the set $\bigcap _{k=1}^{n}S_k^{i_k}$ which is non-empty by (1). Each $B_n$ being a finite intersection of closed sets, $A_k$'s, is closed. By (2) there exists $j\in \mathbb N$ such that $i_j=1$ and diameter$(S_j)\leq1$, and by (3) there exists $l(j)\in \mathbb N$ with
	\[l(j)>j,\ \overline{S}_{l(j)}\subseteq S_k  \text{ and } \text{diameter}(S_{l(j)})\leq \text{diameter}(S_j)/2.\]
	Applying (3) to $l(j)$, there exists $l(l(j))\in \mathbb N$ with
	\[l(l(j))>l(j),\ \overline{S}_{l(l(j))}\subseteq S_l(j)  \text{ and } \text{diameter}(S_{l(l(j))})\leq \text{diameter}(S_l(j))/2.\]
	Continuing this way, we can construct a sequence $\overline{S}_j\supseteq \overline{S}_{l(j)}\supseteq \overline{S}_{l(l(j))}\supseteq \overline{S}_{l(l(l(j)))},\dots$ whose diameter goes to zero. The subsequence $B_{j}\supseteq B_{l(j)}\supseteq B_{l(l(j))}\supseteq B_{l(l(l(j)))},\dots$ of $\{B_n\}_{n\in \mathbb N}$ is such that
	\[B_{j}\supseteq \overline{S}_l(j)\supseteq B_{l(j)}\supseteq \overline{S}_{l(l(j))}\supseteq B_{l(l(j))}\supseteq \overline{S}_{l(l(l(j)))} \supseteq B_{l(l(l(j)))},\dots.\]
	This proves that the diameter of $B_n$ goes to zero as $n$ goes to $\infty$. By Cantor Intersection Theorem, there exists literally one $x\in X$ such that $x\in \cap_{n\in \mathbb N} B_n$. By definition of $B_n$, $B_n\subseteq\cap_{k=1}^{n}U_{k}^{i_k}$ for all $n$, so
	\[x\in \bigcap_{k\in \mathbb N}U_{k}^{i_k},\]
	which means $\varphi(x)=\{i_k\}_{k\in \mathbb N}$. The proof completes.
\end{proof}
\begin{lemma}\label{lemma2}
	The image $\varphi(X)$ is a Borel subset of $\{0,1\}^{\mathbb N}$.
\end{lemma}
\begin{proof} We are going to prove that  $\varphi(X)$ can be written as countable unions and intersections of cylinders in $\{0,1\}^\mathbb N$.
	For a fixed $n\in \mathbb N$ define
	\[S(n)=\big\{(a_1,a_2,\dots,a_n): a_i\in \{0,1\}, \cap_{k=1}^{n}S_k^{a_k}\neq\emptyset\big\}.\]
	Clearly, all $\{i_k\}_{k\in \mathbb N}\in \{0,1\}^\mathbb N$ for which $\cap_{k=1}^{n}S_k^{i_k}\neq \emptyset$ are given by the set
	\[\bigcup_{(a_1,a_2,\dots,a_n)\in S(n)}[a_1,a_2,\dots,a_n].\]
	The set of all $\{i_k\}_{k\in \mathbb N}\in \{0,1\}^\mathbb N$ for which $\cap_{k=1}^{n}S_k^{i_k}\neq \emptyset$ for every $n\in \mathbb N$ is given by
	\begin{equation}\label{eq1}
	\bigcap_{n\in \mathbb N}\ \bigcup_{(a_1,a_2,\dots,a_n)\in S(n)}[a_1,a_2,\dots,a_n].
	\end{equation}
	Let $I$ be the set of all $k\in \mathbb N$ such that diameter($S_k$)$\leq1$. The set of all $\{i_k\}_{k\in \mathbb N}\in \{0,1\}^\mathbb N$ which satisfy property (2) is given by
	\begin{equation}\label{eq2}
	\bigcup_{k\in S}\ \bigcup_{(a_1,a_2,\dots,a_{k-1})\in \{0,1\}^{k-1}}[a_1,a_2,\dots,a_{k-1},1].
	\end{equation}
	Fix $n\in \mathbb N$ and define $M(n)$ by
	\[M(n)=\{i\in \mathbb N:i>n, \overline{S_i}\subseteq S_n,\ \text{diameter}(S_i)\leq \text{diameter}(S_n)/2\}. \]
	The set of all $\{i_k\}_{k\in \mathbb N}\in \{0,1\}^\mathbb N$ which satisfy property (3) for $n$ is the union of
	\[\bigcup_{(a_1,a_2,\dots,a_{n-1})}[a_1,a_2,\dots,a_{n-1},0]\]
	with
	\[
	\bigcup_{i\in M(n)}\ \bigcup_{a_{n+1},a_{n+2},\dots,a_{i-1}}[a_1,a_2,\dots,a_{k-1},1].\]
	The set of all $\{i_k\}_{k\in \mathbb N}\in \{0,1\}^\mathbb N$ which satisfy property (3) is then given by
	\begin{equation}\label{eq3}
	\bigcap_{n\in \mathbb N}\bigg(\bigcup_{(a_1,a_2,\dots,a_{n-1})}[a_1,a_2,\dots,a_{n-1},0]\bigcup \bigcup_{i\in M(n)}\ \bigcup_{a_{n+1},a_{n+2},\dots,a_{i-1}}[a_1,a_2,\dots,a_{k-1},1]\bigg).
	\end{equation}
	By Lemma \ref{lemma1}, $\varphi(X)$ is the intersection of \eqref{eq1}, \eqref{eq2}, and \eqref{eq3}, hence a Borel subset. The proof completes.
\end{proof}

\begin{lemma}
	The map $\varphi: X\to \varphi(X)$ is a measurable bijection with a measurable inverse.
\end{lemma}
\begin{proof} Since the space $X$ is Hausdorff, therefore, for any distinct $x,y\in X$ there exist disjoint $S_k, S_l \in U$ such that $x\in S_k$ and $y\in S_l$. This proves $\varphi$ is injective. Since $\varphi^{-1}$ is well-behaved to union and intersection, and the Borel sigma  algebra $\widehat{\mathcal B}$ is generated by the cylinders $[a_1,a_2,\dots, a_n]$, $n\geq 0$, $a_i \in \{0,1\}$, so to prove $\varphi$ is measurable it suffices to show that $\varphi^{-1}([a_1,a_2,\dots, a_n])$
	is measurable for every $[a_1,a_2,\dots, a_n]\in \widehat{\mathcal B}$. Clearly, for any $n\geq0$ and $a_1,a_2,\dots, a_n \in \{0,1\}$
	\[\varphi^{-1}([a_1,a_2,\dots, a_n])=\bigcap_{i=1}^{n}U_i^{a_i}.\]
	This proves $\varphi$ is measurable.
	
	Note that the elements of $\mathcal{A}$ are finite intersections and the complements of the elements of $U$. Now, $\varphi^{-1}$ being injective is well-behaved to intersection and unions, and since the Borel sigma  algebra $\widehat{\mathcal A}$ is generated by $\mathcal{A}$, so to prove $\varphi^{-1}$ is measurable it suffices to show that $\varphi(\cap_{i=1}^{n}U_i^{a_i})$ is measurable for every $n\geq0$ and $a_1,a_2,\dots, a_n \in \{0,1\}$. Clearly, for any $n\geq0$ and $a_1,a_2,\dots, a_n \in \{0,1\}$
	\[\varphi(\bigcap_{i=1}^{n}U_i^{a_i})=[a_1,a_2,\dots, a_n]\bigcap\varphi(X).\]
	This together with Lemma \ref{lemma2} proves $\varphi^{-1}$ is measurable.
\end{proof}

Now, with the above three lemmas, we can easily prove the Extension Theorem,  more precisely we can prove the extension of every finite additive function from \( {\mathcal A} \to [0,1]\) to a probability measure on the sigma algebra \(\widehat{\mathcal A}\).
\begin{proof}[Proof of Theorem \ref{ext}]
	Define $\psi:\mathcal{{B}}\to [0,1]$ by
	\begin{equation}\label{eq5}
	\psi(B)=\mu(\varphi^{-1}(B)), \ \ B\in \mathcal{{B}}.
	\end{equation}
	Clearly, $\psi$ is finite additive. Also, the algebra $\mathcal{{B}}$ is compact, $\psi$ is $\sigma$-additive. Let $\widehat{\psi}: \widehat{\mathcal B}\to [0,1]$ denotes the extension of $\psi$ to the probability measure. We prove that $\widehat{\mu}:\widehat{\mathcal A}\to [0,1]$ defined by
	\[\widehat{\mu}(A)=\widehat{\psi}(\varphi(A))\]
	extends $\mu$ to a probability measure. Clearly, $\widehat{\mu}$ is a probability measure. Let $\digamma$ be a cover of $\varphi (X)$ by cylinders, then
	\[\widehat{\psi}(\bigcup_{D\in \digamma} D)=\mu(\bigcup_{D\in \digamma}\varphi^{-1}(D))=\mu(X)=1.\]
	Since $\widehat{\psi}$ is regular, so
	\begin{equation}\label{eq4}
	\widehat{\psi}(\varphi(X))=\inf\Big\{\widehat{\psi}(\bigcup_{D\in \digamma} D): \digamma \text{ is a cover of X by cylinders}\Big\}=1.
	\end{equation}
	For every $n\in \mathbb N$ and $a_1,a_2,\dots, a_n \in \{0,1\}$, we have
	\begin{equation*}
	\begin{split}
	\widehat{\mu}(\bigcap_{i=1}^{n}U_i^{a_i})& = \widehat{\psi}(\varphi(\bigcap_{i=1}^{n}U_i^{a_i}))\\
	&=\widehat{\psi}([a_1,a_2,\dots, a_n]\bigcap\varphi(X))\\
	&=\psi([a_1,a_2,\dots, a_n]\bigcap\varphi(X))\\
	&=\mu(\varphi^{-1}([a_1,a_2,\dots, a_n])), \ \text{ (by \eqref{eq4} and \eqref{eq5})}\\
	&=\mu(\bigcap_{i=1}^{n}U_i^{a_i}).
	\end{split}
	\end{equation*}
	This proves $\widehat{\mu}$ to be an extension of $\mu$ to a probability measure. This completes the proof of the theorem.
\end{proof}

\section{Proof of Rokhlin disintegration theorem}\label{sec:rokhlin}
To prove Theorem \ref{Rokhlin}, we first need to do the following construction of conditional measures.
Recall that \((X, \widehat{\mathcal A}, \mu)\) is a probability space with \(\mathbb P\) as its measurable partition that is, there exists a sequence \(\mathbb P_n\) of countable partitions of \(X\) such that \(\mathbb P_{n}\preceq \mathbb P_n\) for all \(n\in \mathbb N\) and \(\mathbb P= \bigcap_{n\in \mathbb N}\mathbb P_n\).

Let $\phi:X\to \mathbb R$ be a bounded measurable function. Consider the sequence $\alpha_n(\phi,\cdot):X\to \mathbb R$ of functions defined by
\[
\alpha_n(\phi,x)=
\begin{cases}
\frac{1}{\mu(P_{n}(x))}{\displaystyle\int_{P_{n}(x)}}\phi d\mu & \text{if $\mu(P_{n}(x))>0$}\\
\vspace{.2cm}\\
0 & \text{if $\mu(P_{n}(x))=0$}.
\end{cases}\]

\begin{lemma}\label{lem lim}There exists a set \(  x\in X_{\phi} \subset X  \) with \(  \mu( X_{\phi})=1  \) such that for each $x\in X_{\phi}$, the limit $\alpha (\phi,x):=\lim \alpha_n(\phi,x)$ exists.
\end{lemma}

\begin{proof}
	We know that $\phi$ is integrable.
	So for each $n\in \mathbb{N}$, the function $\alpha(\phi,\cdot)$ is well defined. Moreover, each $\alpha_n(\phi,\cdot)$ assumes countable number of values because each $\mathbb{P}_n$ is countable and $\alpha_n(\phi,\cdot)$ is constant on each $P_n\in \mathbb{P}_n$, therefore  $\alpha_n(\phi,\cdot)$ also measurable. For all $n\in \mathbb N$ and $x\in X$
	\[|\alpha_n(\phi,x)|\leq \frac{1}{\mu(P_{n}(x))}\int_{P_{n}(x)}|\phi| d\mu\leq \sup |\phi|,\]
	that is, the sequence $\alpha_n(\phi,\cdot)$ is uniformly bounded. Therefore, for each $x\in X$ the $\liminf \alpha_n(\phi,x)$ and $\limsup \alpha_n(\phi,x)$ exist and are finite. For $\beta,\rho \in \mathbb{Q}$ define $X(\beta,\rho)$ by
	\[X(\beta,\rho)=\{x\in X: \liminf \alpha_n(\phi,x)<\beta<\rho <\limsup \alpha_n(\phi,x)\}.\]
	To prove $\mu(\cup_{\beta,\rho \in \mathbb{Q}}X(\beta,\rho))=0$ which is equivalent to proving $\mu(X(\beta,\rho))=0$ for every $\beta,\rho \in \mathbb{Q}$. Fix $\beta,\rho \in \mathbb{Q}$. For a given $x\in X(\beta,\rho)$, take any two sequences $\{a^{x}_i\}$ and $\{b^{x}_i\}$ such that $a^{x}_i<b^{x}_1<a^{x}_2<b^{x}_2<\dots$ with
	\[\alpha_{a_i^{x}}(\phi,x)<\beta \text{ and }  \alpha_{b_i^{x}}(\phi,x)>\rho \text{ for every } i\geq1.\]
	Let $A_i(x)=P_{a^{x}}(x)$ and $B_i(x)=P_{b^{x}}(x)$. Define $A_i$ and $B_i$ by
	\begin{equation}\label{Ai}
	A_i=\bigcup_{x\in X(\beta,\rho)}A_i(x), \text{ and } B_i=\bigcup_{x\in X(\beta,\rho)}B_i(x),
	\end{equation}
	respectively. Since $a^{x}_{i}<b^{x}_{i}$ for every $i\in \mathbb N$, the partition $\mathbb{P}_{b^{x}_{i}}$ is a refinement of $\mathbb{P}_{a^{x}_{i}}$.
	So for every $i\in \mathbb N$, $X(\beta,\rho)\subseteq A_{i+1}\subseteq B_{i}\subseteq A_{i}$ and therefore
	\[X(\beta,\rho) \subseteq \bigcap_{i\in \mathbb N}A_i=\bigcap_{i\in \mathbb N}B_i.\]
	Because the sequence $\mathbb{P}_{n}$ of partitions is increasing, we can assume the sets $A_{i}(x)$ that form $A_{i}$ in \eqref{Ai} to be disjoint. Also observe that
	\[\int\phi d\mu=\sum_{P\in \mathbb{P}_{n}}\int_{P}\phi d\mu=\sum_{P\in \mathbb{P}_{n}}\mu(P)E(\phi,P)=\int\alpha_{n}(\phi) d\mu\]
	Using these facts, we have
	\begin{equation*}
	\int_{A_i}\phi d\mu=\sum_{A_{i}(x)}\int_{A_{i}(x)}\phi d\mu
	=\sum_{A_{i}(x)}\int_{A_{i}(x)}\alpha_{a^{x}_{i}}(\phi) d\mu
	\leq\sum_{A_{i}(x)}\beta \mu(A_{i}(x))
	=\beta \mu(A_{i}).
	\end{equation*}
	Similarly,
	\begin{equation*}
	\int_{B_i}\phi d\mu=\sum_{B_{i}(x)}\int_{B_{i}(x)}\phi d\mu
	=\sum_{B_{i}(x)}\int_{B_{i}(x)}\alpha_{b^{x}_{i}}(\phi) d\mu
	\geq\sum_{B_{i}(x)}\rho \mu(B_{i}(x))
	=\rho \mu(B_{i}).
	\end{equation*}
	$\phi\geq0$ and $B_i\subseteq A_i$ imply
	\[\beta \mu(A_{i})\geq \int_{A_i}\phi d\mu\geq \int_{B_i}\phi d\mu\geq \rho \mu(B_{i}),\]
	for every $i\in \mathbb N$. Taking $i\to \infty $ we get
	\[(\beta-\rho) \mu(\cup_{x\in X(\beta,\rho)}A_i(x))\geq0,\]
	which is true if and only if $\mu(\cup_{x\in X(\beta,\rho)}A_i(x))=0$. Consequently, $\mu(X(\beta,\rho))=0$ for every pair $\beta,\rho \in\mathbb{Q}$. This completes the proof for $\phi\geq0$.
	The general conclusion follows from the fact that every measurable bounded function $\phi$ can be written as the difference of two non--negative bounded measurable functions $\phi^{\pm}=\max\{0,\pm\phi\}$.
\end{proof}

We will apply Lemma \ref{lem lim} in the particular case where \(  \varphi  \) is the characteristic function of a set \(  A  \), in which case the following limit exists:
\begin{equation}\label{eq7}
\alpha(x,\chi_ {A})=\lim_{n\to \infty}\frac{\mu(P_{n}(x)\cap A)}{\mu(P_{n}(x))}.
\end{equation}

\begin{lemma}\label{lem meas}
	The function $\alpha(x,\chi_A)$ is measurable and is constant on each $P\in \mathbb{P}$.
	Moreover
	\[\int\chi_A d \mu=\int\alpha(\chi_A)d\mu\]
\end{lemma}
\begin{proof}
	
	The function $\alpha(\phi)$ is measurable because it is the pairwise limit of the sequence of measurable functions $\alpha_{n}(\phi)$. For a given $P\in \mathbb{P}$ there exists, by definition, a sequence $P_n$ such $P_n\in \mathbb{P}_n$ for all $n\in \mathbb N$ and $P=\cap_{n\in \mathbb N}P_n$. Since $\alpha{_n}(\phi)$ is constant on $P_k$ for all $k\geq n$, therefore, on $P=\cap_{n\in \mathbb N}P_n$. Hence $\alpha(\phi)$ is constant on $X_\phi\cap P$.
	By Dominated Convergence Theorem, we have
	\[\int \alpha(\phi) d\mu=\lim_{n\to \infty}\int \alpha_n(\phi)d\mu=\int \lim_{n\to \infty} \alpha_n(\phi) d\mu=\int \phi d\mu.\qedhere\]
\end{proof}
We can now construct a family of finite additive measures on \(\mathcal A\).
Let $\mathbb{P}_{A}$ be all partition elements $P$ that intersects $X_{\chi_A}$, then $\widehat{\mu}(\mathbb{P}_{A})=1$.
Define $E(A,.):\mathbb{P}_{A} \to \mathbb{R}$ by $E(A,P)=\alpha(x,\chi_ {A})$, where $x\in X_{\chi_A}\cap P$. Note that  $E(A,P)$ is a constant for each $P\in \mathbb{P}_{A}$.
Clearly as by Lemma \ref{lem meas}, $\alpha(x,\chi_A)$ is measurable, so the mapping $E(A,.)$ is measurable and also
\begin{equation}\label{eq6}
\int \phi d\mu=\int \alpha(\chi_ {A})d\mu=\int E(A) d\widehat{\mu}.
\end{equation}
Define $\mathbb{P}^\prime$ by
\[\mathbb{P}^\prime=\bigcap_{A\in \mathcal{\widehat{A}}}\mathbb{P}_{A}.\]
Since $\widehat{\mu}(\mathbb{P}_{A})=1$ for every $A\in \mathcal{\widehat{A}}$ and the intersection is countable, so $\widehat{\mu}(\mathbb{P}^\prime)=1$. For $P\in \mathbb{P}^\prime$ define $\mu_{P}:\mathcal{\widehat{A}}\to [0,1]$ by
\[\mu_{P}(A)=E(A,P).\]
Clearly, $\mu_{P}(X)=E(X,P)=1$ and $\mu_{P}(\emptyset)=E(\emptyset,P)=0$.
Claim: $\mu_P$ is a finitely additive function.
\\ If $A, B\in \mathcal{A}$ are two disjoint sets then
\[\alpha(x,\chi_ {A})=\lim_{n\to \infty}\frac{\mu(P_{n}(x)\cap A)}{\mu(P_{n}(x))},\ \ \alpha(x,\chi_ {B})=\lim_{n\to \infty}\frac{\mu(P_{n}(x)\cap B)}{\mu(P_{n}(x))}\]
exists by \eqref{eq7}, therefore
\begin{equation*}
\begin{split}
\mu_{P}(A\cup B)=E(A\cup B,P)& = \lim_{n\to \infty}\frac{\mu(P_{n}(x)\cap (A\cup B))}{\mu(P_{n}(x))}\\
&=\lim_{n\to \infty}\Bigg\{\frac{\mu(P_{n}(x)\cap A)}{\mu(P_{n}(x))}+\frac{\mu(P_{n}(x)\cap B)}{\mu(P_{n}(x))}\Bigg\}\\
&=\lim_{n\to \infty}\frac{\mu(P_{n}(x)\cap A)}{\mu(P_{n}(x))}+\lim_{n\to \infty}\frac{\mu(P_{n}(x)\cap B)}{\mu(P_{n}(x))}\\
&=E(A,P)+E(B,P)\\
&=\mu_{P}(A)+\mu_{P}(B).
\end{split}
\end{equation*}
By Theorem \ref{ext}, $\mu_{P}$ can be extended to a probability measure on $X$. We still denote this extension by $\mu_{P}$. Thus, we have constructed a family of measures, $\{\mu_{P}:P\in \mathbb{P}\}$.

\begin{lemma}
	The family of measures $\{\mu_{P}:P\in \mathbb{P}\}$   is a disintegration of $\mu$ with respect to $\mathbb{P}$.
\end{lemma}
\begin{proof}  Let $\digamma$ be the class of all measurable subsets of $X$ that satisfies the property (2) of the Definition \ref{dis}. It is enough to prove that $\digamma$ is a monotonic class containing the algebra $\mathcal{\widehat{A}}$.
	
	By equation \eqref{eq6}, for each $A\in \mathcal{\widehat{A}}$ the map $P\to \mu_{P}(A)$ is measurable and
	\[
	\mu(A)=\int \chi_ {A} d\mu=\int \mu_{P}(A) d\widehat{\mu}.
	\]
	This proves that $\mathcal{\widehat{A}}$ is contained in $\digamma$. Let $\{A_k\}_{k\in N}$ be an increasing, ordered by inclusion, sequence in $\digamma$. Then for each $k$ the sequence of map $P\to \mu_{P}(A_k)$ is measurable and
	\[
	\mu(A_k)=\int \chi_ {A_k} d\mu=\int \mu_{P}(A_k) d\widehat{\mu}.
	\]
	Let $A=\cup_{k\in \mathbb N}A_k$. The map
	\[P\to \mu_{P}(A)=\sup_{k\in \mathbb N}\mu_{P}(A_k),\]
	being the point wise limit of measurable functions, is measurable. Also by Monotone Convergence Theorem, we have
	\begin{equation*}
	\begin{split}
	\mu(A_k)&=\lim_{k \to \infty}\mu(A_k)=\lim_{k \to \infty}\int \chi_ {A_k} d\mu=\lim_{k \to \infty} \int \mu_{P}(A_k) d\widehat{\mu}\\
	&=\int \lim_{k \to \infty}\mu_{P}(A_k) d\widehat{\mu}=\int \mu_{P}(A) d\widehat{\mu}.
	\end{split}
	\end{equation*}
	Hence $A\in \digamma$. Similarly, If $\{A_k\}_{k\in \mathbb N}$ is a decreasing, ordered by inclusion, sequence in $\digamma$. Then for each $k$ the sequence of map $P\to \mu_{P}(A_k)$ is measurable and
	\[
	\mu(A_k)=\int \chi_ {A_k} d\mu=\int \mu_{P}(A_k) d\widehat{\mu}.
	\]
	Let $B=\cap_{k\in \mathbb N}A_k$. Then
	\[P\to \mu_{P}(B)=\lim_{k \to \infty}\mu_{P}(A_k)=\inf_{k\in N}\mu_{P}(A_k)\]
	being the point wise limit of measurable function is measurable. Also by  Monotone Convergence Theorem, we have
	\[
	\mu(B)=\int \mu_{P}(B) d\widehat{\mu}.
	\]
	Thus, $B\in \digamma$. This proves that $\digamma$ is a monotonic class containing $\mathcal{\widehat{A}}$. Hence, $\{\mu_{P}:P\in \mathbb{P}\}$ is a disintegration of $\mu$ with respect to $\mathbb{P}$.
\end{proof}

\section{Proof of ergodic decomposition theorem}

Recall that \((X, \widehat{\mathcal A},\mu)\) is a Borel probability space, \(X\) is a completely separable metric space, \(f:X\to X\) is a measurable transformation and \(\mathbb P_f\) is the dynamical partition of \(X\) with respect to \(f\).
Now, having all the tools we require to prove the main result, that is, Theorem \ref{ergodic}, we are proving it within the following lemmas.
\begin{lemma}\label{partn}
	The dynamical partition $ \mathbb{P}$ is a measurable partition.
\end{lemma}
\begin{proof}
	Let $\{q_j: j \in \mathbb{N}\}$ be the set of rationals in [0, 1] and ${\mathcal{A}}=\{A_k; k \in \mathbb{N}\}$. For given $n \in \mathbb{N}$ consider the partition of the interval [0, 1] into intervals determined by the points $q_1, q_2,\ldots,q_n$.Define a relation $\backsim_n$ on $X_0$ by $x\backsim_n y$ if and only if for every $k= 1,2,3,\ldots,n,$ both $\tau(x, A_k)$ and \(\tau(y, A_k)\) belong to the same interval of the partition. Clearly, it is an equivalence relation on $X_0$. Let $\mathbb{P}_n$ be the partition determined by this relation. For each $n \in \mathbb{N}$, the partition $\mathbb{P}_n$ is finite because there are only finite number of ways, in fact $2^n$ ways, by which two elements $x, y \in X_0 $ can be related by $\backsim_n$. Also $\mathbb{P}_n \preceq \mathbb{P}_{n+1}$ for all $n \in \mathbb{N}$ is true by the definition of the relations determining $\mathbb{P}_n$ and $\mathbb{P}_{n+1}$. If for some $j, k \in \mathbb{N}$, $\tau(x, A_k)> q_j$ and $\tau(y, A_k)\leq q_j$, then $x, y$ are not in the same element $P\in \mathbb{P}_j$ and hence of $\bigwedge_{n \in \mathbb{N}} \mathbb{P}_n$. So if $x,y$ are in the same element of $\bigwedge_{n \in \mathbb{N}} \mathbb{P}_n$, then $\tau(x,A_k)= \tau(y,A_k)$ for every $k \in \mathbb{N}$ and also conversely. This proves that
	\begin{equation*}
	\mathbb{P}= \bigwedge_{n \in \mathbb{N}} \mathbb{P}_n. \qedhere
	\end{equation*}
\end{proof}
Now,  for Borel probability space $(X,\widehat{\mathcal{A}},\mu)$ with the dynamical partition $\mathbb{P}$, the Theorem \ref{Rokhlin} gives us a unique disintegration of $\mu$, say, $\{\mu_P:P\in \mathbb{P}\}$. Evidently the elements of partition $\mathbb{P}$ of $X$ are the union of orbits of elements of X, which induces several dynamical properties to $\mu_P$ for $\widehat{\mu}$-almost every $P$.
\begin{lemma}
	The conditional measure $\mu_P$ is ergodic for $\widehat{\mu}$-almost every $P$.
\end{lemma}
\begin{proof}
	Note that to prove that $\widehat{\mu}$- almost every $\mu_P$ is ergodic is equivalent to proving that for $\mu_P$-almost every $x\in X_0\cap P$, the function $\tau(x, E)$ is constant for every measurable set $E$ and every $P\in \mathbb{P}$. Fix $P \in \mathbb{P}$ and let $\mathcal{F}$ be the class of measurable sets for which the above property holds. Clearly $\mathcal{F}$ contains $A_U$ by the definition of $\mathbb{P}$. We prove that $\mathcal{F}$ is a monotonic class. Note that for any $E_1, E_2 \in \mathcal {F}$ such that $E_1 \subset E_2$, we have
	\begin{equation}
	\tau(x, E_2\backslash E_1) = \tau(x, E_2)- \tau(x, E_!),
	\end{equation}
	which shows that $\tau(x, E_2\backslash E_1)$ is well defined and constant on $X_0\cap P$, so $E_2\backslash E_1 \in \mathcal{F}$. In particular $E^C \in \mathcal{F}$ for every $E \in \mathcal{F}$. Also if $E_k$ is a sequence of pairwise disjoint sets in $\mathcal{F}$ then
	\begin{center}
		$\tau(x,\bigcup\limits_{k\in \mathbb{N}}E_k) = \sum\limits_{k \in \mathbb{N}}\tau(x, E_k)$
	\end{center}
	is constant on $X_0 \cap P$. This proves that $\bigcup\limits_{k \in \mathbb{N}}E_k\in \mathcal{F}$. From these two observations it follows that $\mathcal{F}$ is a monotonic class as follows: Let $A_i$ and $B_i$ be two sequences in $\mathcal{F}$ such that $B_i\subset B_{i+1}$ and $A_i\supset A_{i+1}$ for all $i\in \mathbb{N}$. Then using (1), we have
	\begin{center}
		$\bigcup\limits_{i \in \mathbb{N}}B_i= B_1 \bigcup\limits_{i \in \mathbb{N}}(B_{i+1}\backslash B_i)\in \mathcal{F}$ and $\bigcap\limits_{i \in \mathbb{N}}A_i=(\bigcup\limits_{i \in \mathbb{N}}A_i)^C\in \mathcal{F}$.
	\end{center}
	This proves that $\mathcal{F}$ is a monotoic class. Now, by the Monotonic Class Theorem, we get that $\mathcal{F}=\widehat{\mathcal{A}}$ the result follows.
\end{proof}

\begin{lemma}
	If $\mu$ is non-singular with respect to f then so is $\mu_P$ for $\widehat{\mu}$-almost every \(P\).
\end{lemma}
\begin{proof}
	Let us suppose to the contrary that there exists $A\subset \widehat{\mathcal{A}}$ such that $\widehat{\mu}(A)>0$ and for all $P\in A$, $\mu_P$ is singular with respect to $f$, that is, there exists $M_P\subset X$ for all $P \in A$ such that
	\begin{equation}
	\mu_P(M_P)=0\ \text{but}\  \mu_P(f^{-1}(M_P))>0.
	\end{equation}
	Note that since $\mu_P(P)=1$ for all $P\in \mathbb{P}$, therefore, for $P\in A$, $\mu_P(f^{-1}(M_P))>0$ implies $f^{-1}(M_P)\subset P$, which in turn implies
	\begin{equation}
	M_P\subset f(P) \subset P.
	\end{equation}
	Let $E=\bigcup\limits_{\mu_P\in A}M_P$, so we have for $P\in A$,
	\begin{equation}
	\mu_P(E)=\mu_P(\bigcup\limits_{\mu_P\in A}M_P)=0,
	\end{equation}
	so for $P\in A$, we have
	\begin{equation}
	\mu_P(f^{-1}(E))>0
	\end{equation}
	and since $\mu_P$'s are mutually singular, so for $P\in A^C$, we have
	\begin{equation}
	\mu_P(f^{-1}(E))=0.
	\end{equation}
	Now, using (6), we have $\mu(E)=\int\mu_P(E)d\widehat{\mu}(P)=0$. As $\mu$ is non-singular with respect to $f$, therefore we must have $\mu(f^{-1}(E))=0$. Now,
	\begin{align*}
	\mu(f^{-1}(E))&=\int\mu_P(f^{-1}(E))d\widehat{\mu}(P)
	\\&=\int_{A}\mu_P(f^{-1}(E))d\widehat{\mu}(P) + \int_{A^C}\mu_P(f^{-1}(E))d\widehat{\mu}(P)
	\\&>0
	\end{align*}
	which is a contradiciton to the non-singularity of $\mu$, and hence the result.	  	
\end{proof}
\begin{lemma}
	If $\mu$ is $f$-invariant then so is $\mu_P$ for $\widehat{\mu}$-almost every \(P\).
\end{lemma}
\begin{proof}
	We need to prove that $\widehat{\mu}$-almost every $\mu_P$ is fixed by the pull-back by $f$, that is, $f_*\mu_P=\mu_P$.
	Every $P$ being the union of some orbits, is an $f$-invariant set. Therefore
	\begin{equation}
	f_*\mu_P(P)=\mu(f^{-1}(P))=\mu(P)=1.
	\end{equation}
	Note that the measurability of the maps $f$ and $P \rightarrow \mu_P(E)$ for every measurable set $E$, imply the measurabikity of the map $P \rightarrow f_*\mu_P(E)=\mu_P(f^{-1}(P))$.
	\\ The $f$-invariance of $\mu$ implies
	\begin{equation}
	\mu(E)= \mu(f^{-1}(E))=\int\mu_P(f^{-1}(E))d\widehat{\mu}(P).
	\end{equation}
	which implies that the family of probability measures $\{f_*\mu_P:P\in \mathbb{P}\}$ also disintegrates $\mu$ with respect to $\mathbb{P}$. By the uniqueness of disintegration for a specific partition, we have that $f_*\mu_P=\mu_P$ for $\widehat{\mu}$-almost every $P\in \mathbb{P}$ and hence the result.
\end{proof}
\


\begin{thebibliography}{00}
\bibitem{RO}Rohlin, V. A., \emph{On the Fundamental Ideas of Measure Theory}, American Mathematical Society Translation, vol. 71. American Mathematical Society, Providence, RI (1952).
\bibitem{W} Viana, M., Oliveira, K., (2016), \emph{Foundations of Ergodic Theory}, Cambridge Studies in Advanced Mathematics, Cambridge: Cambridge University Press.
\end{thebibliography}
\end{document}